\documentclass[11pt,a4paper,french,english]{article}
\usepackage[top=3cm, bottom=3cm, left=3cm, right=3cm]{geometry}
\usepackage[T1]{fontenc}
\usepackage{lmodern}
\usepackage{babel}
\usepackage[latin1]{inputenc}
\usepackage{csquotes}
\usepackage{amsmath}
\usepackage{amsthm}
\usepackage{amsfonts}
\usepackage{amssymb}
\usepackage{graphics}
\usepackage{float}
\usepackage[hang,flushmargin]{footmisc} 

\usepackage{amsmath,amsthm,amssymb,graphicx, multicol, array}
\usepackage{enumerate}
\usepackage{enumitem}
\usepackage{xcolor}
\usepackage{multicol}

\usepackage[pagebackref,bookmarks,colorlinks,breaklinks,linktoc=page]{hyperref}
\hypersetup{linkcolor=blue,citecolor=red,filecolor=blue,urlcolor=blue} 

\usepackage{tabto}

\numberwithin{equation}{section}

\usepackage{tikz}
\usepackage{pgfplots}

\usetikzlibrary{matrix}

\usepackage{mathrsfs}

\DeclareMathOperator{\ord}{ord}

\newtheorem{lem}{Lemma}[section]
\newtheorem{theorem}{Theorem}[section]
\newtheorem{lemma}{Lemma}[section]

\newtheorem{exa}{Example}[section]

\newtheorem{exe}{Exercise}[section]

\newcommand{\N}{\mathbb{N}}
\newcommand{\Z}{\mathbb{Z}}

\newcommand{\C}{\mathbb{C}}
\newcommand{\F}{\mathbb{F}}

\usepackage{fancyhdr}
\title{Least Consecutive Pair of Primitive Roots}
\date{}
\author{N. A. Carella}

\begin{document}
\maketitle

\begin{abstract}
Let $p>1$ be a large prime number and let $x=O((\log p)^2(\log\log p)^5$ be a real number. It is proved that the least consecutive pair of primitive roots $u\ne\pm1, v^2$ and $u+1$ satisfies the upper bound $u\ll x$ in the prime field $\mathbb{F}_p$. 
\let\thefootnote\relax\footnote{ \today \date{} 
	}
\end{abstract}


\section{Introduction }\label{S3355B}
There are many results on the existence of pairs of consecutive primitive roots $u$ and $u+1$ modulo a prime $p$. The general theory on the existence of $k$-tuples of consecutive primitive roots 
\begin{equation}
u,\quad u+1,\quad u+2,\quad  \ldots, \quad u+k-1,
\end{equation}
in the prime finite field $\F_p$ for $k\ll\log p$ has a large and well established literature. The current literature covers various results on the existence of consecutive pairs of primitive roots and $k$-consecutive primitive roots, see \cite{CL1956}, \cite{VE1968}, \cite{SM1975}, \cite{TT2013}, \cite{LD2015}, \cite{CN2019}, et alii. In \cite{CZ1998} it is shown that consecutive pairs of primitive roots have a Poisson distribution in certain sequence of primes. In {\color{red}\cite[Theorem 1.1]{LR2022}} some results on the existence of $k$-tuples of primitive roots in arithmetic progressions
\begin{equation}
\alpha,\quad \alpha+1\cdot \beta,\quad \alpha+2\cdot \beta,\quad  \ldots,\quad \alpha+(k-1)\cdot \beta,
\end{equation}
in the finite field $\F_{q^n}$, where $\alpha,\beta\in \F_q$, $q$ is a prime power and $n\geq2$, are proved. The asymptotic formula for the counting function for the number of $k$-tuples of consecutive primitive roots 
\begin{align}\label{eq3355NE.100d}
N_k(p)&=c_k(p)\left (\frac{\varphi(p-1)}{p-1}\right )^k\cdot (p-1)+O(p^{1/2+\varepsilon}),
\end{align}
where $c_k(p)>0$ is a density constant, is nearly the same as the product of $k$ independent random primitive roots in the finite field $\mathbb{F}_p$.  \\

This note explores a new topic: the least consecutive pairs of primitive roots, and more generally the least $k$-tuples of consecutive primitive roots. 

\begin{theorem} \label{thm3355DD.850A}\hypertarget{thm3355DD.850A} Let $p>1$ be a large prime number and let $x=O(\log p)^2(\log\log p)^5$ be a small real number. Then there exists a pair of consecutive primitive root $$2\leq u\quad \text{ and }\quad u+1\ll x$$ 
in the prime finite field $\F_p$, unconditionally.
\end{theorem}	
\vskip .1 in
There is no literature on this topic. Thus, it is unknown if this is the best possible upper bound for the least consecutive pair of primitive roots. The proof of this result seamlessly extends to the $k$-tuple case. 

\begin{theorem} \label{thm3355DD.850K}\hypertarget{thm3355DD.850K} Let $p>1$ be a large prime number and let $x=O(\log p)^k(\log\log p)^{k+3}$ be a small real number. If $ k\ll \log p$, then there exists a $k$-tuple of consecutive primitive root $$u,\quad u+1,\quad u+2,\quad  \ldots, \quad u+k-1\ll x$$ 
in the prime finite field $\F_p$, unconditionally.
\end{theorem}

\hyperlink{thm3355DD.850A}{Theorem} \ref{thm3355DD.850A} appears in \hyperlink{S3355M-T}{Section} \ref{S3355M-T}. This result reduces the  search algorithm in finite field $\mathbb{F}_p$ for small pair of consecutive primitive roots from exponential time complexity to polynomial time complexity if the prime factorization of the totient $p-1$ is provided.

\section{Representations of the Characteristic Function}\label{S9955D}\hypertarget{S9955D}
The \textit{multiplicative order} of an element in a finite field is defined by $\ord_p u=\min\{k:u^k\equiv 1 \bmod p\}$. An element $u\ne \pm1,v^2$ is called a primitive root if $\ord_p u=p-1.$ The characteristic function \(\Psi :G\longrightarrow \{ 0, 1 \}\) of primitive elements is one of the standard analytic tools employed to investigate the various properties of primitive roots in cyclic groups \(G\). Many equivalent representations of the characteristic function $\Psi $ of primitive elements
are possible, a few are investigated here. 
\subsection{Divisor Dependent Characteristic Function}		
The divisor dependent characteristic function was developed about a century ago. The authors in \cite{DH1937} and \cite{WR2001} attributed this formula to Vinogradov, and other authors have attributed it to Landau, see {\color{red}\cite[Theorem 496]{LE1927}}, {\color{red}\cite[p.\; 258]{LN1997}}, et alia. This characteristic function detects the order of an element by means of the divisors of the totient $p-1$. The precise description is stated below.

\begin{lemma} \label{lem9955.200D} \hypertarget{lem9955.200D} Let \(p\geq 2\) be a prime and let \(\chi\) be a multiplicative character of order $\ord  \chi =d$. If \(u\in
	\mathbb{F}_p\) is a nonzero element, then
	\begin{equation}
		\Psi (u)=\frac{\varphi(p-1)}{p-1}\sum _{d\mid p-1} \frac{\mu(d)}{\varphi(q)}\sum _{\ord \chi =d} \chi(u)
		=\left \{
		\begin{array}{ll}
			1 & \text{ if } \ord_p (u)=p-1,  \\
			0 & \text{ if } \ord_p (u)\neq p-1, \\
		\end{array} \right .\nonumber
	\end{equation}
	where $\mu:\N\longrightarrow \{-1,0,1\}$ is the Mobius function.
\end{lemma}	

\subsection{Divisorfree Characteristic Function}	
A new \textit{divisors-free} representation of the characteristic function of primitive element is developed here. It detects the order \(\text{ord}_p
(u) \geq 1\) of the element \(u\in \mathbb{F}_p\) by means of the solutions of the equation \(\tau ^n-u=0\) in \(\mathbb{F}_p\), where
\(u,\tau\) are constants, and $n$ is a variable such that \(1\leq n<p-1, \gcd (n,p-1)=1\). 

\begin{lemma} \label{lem9955.200A} \hypertarget{lem9955.200A} Let \(p\geq 2\) be a prime and let \(\tau\) be a primitive root mod \(p\) and  let \(\psi \neq 1\) be a nonprincipal additive character of order $\ord  \psi =p$. If \(u\in
	\mathbb{F}_p\) is a nonzero element, then
	\begin{equation}
		\Psi (u)=\sum _{\gcd (n,p-1)=1} \frac{1}{p}\sum _{0\leq s\leq p-1} \psi \left ((\tau ^n-u)s\right)
		=\left \{
		\begin{array}{ll}
			1 & \text{ if } \ord_p (u)=p-1,  \\
			0 & \text{ if } \ord_p (u)\neq p-1. \\
		\end{array} \right .\nonumber
	\end{equation}
\end{lemma}	
\begin{proof}[\textbf{Proof}] As the index \(n\geq 1\) ranges over the integers relatively prime to \(p-1\), the element \(\tau ^n\in \mathbb{F}_p\) ranges over the primitive roots modulo $p$. Ergo, the equation \(\tau ^n- u=0\) has a solution if and only if the fixed element \(u\in \mathbb{F}_p\) is a primitive root. Next, replace \(\psi (z)=e^{i 2\pi  k z/p }\) to obtain 
	
	\begin{equation}
		\sum_{\gcd (n,p-1)=1} \frac{1}{p}\sum_{0\leq s\leq p-1} e^{\frac{i 2\pi  (\tau ^n-u)s}{p} }=
		\left \{\begin{array}{ll}
			1 & \text{ if } \ord_p (u)=p-1,  \\
			0 & \text{ if } \ord_p (u)\neq p-1. \\
		\end{array} \right.
	\end{equation}
	\vskip .1 in	 
	This follows from the geometric series $\sum_{0\leq n\leq N-1} w^n =(w^N-1)(w-1),w\ne1$ applied to the inner sum.
\end{proof}

Another approach to the construction of divisorfree characteristic function is via the discrete logarithm $\log_{\tau}:\F_p^{\times}\longrightarrow \F_p$ in finite field $\F_p$.

\begin{lem} \label{lem9955FF.200DFL} \hypertarget{lem9955FF.200DFL}  Let \(p\geq 2\) be a prime and let \(\tau\) be a primitive root mod \(p\) and  let \(\psi \neq 1\) be a nonprincipal additive character of order $\ord  \psi =p$. If \(u\in
	\mathbb{F}_p\) is a nonzero element, then
	\begin{equation}
		\Psi (u)=\sum _{\substack{1\leq s< p-1\\\gcd (s,p-1)=1}} \frac{1}{p}\sum _{0\leq t< p} \psi \left ((s-\log_{\tau}u)t\right)
		=\left \{
		\begin{array}{ll}
			1 & \text{   \normalfont if } \ord_p (u)=p-1,  \\
			0 & \text{   \normalfont if } \ord_p (u)\neq p-1. \\
		\end{array} \right .\nonumber
	\end{equation}
\end{lem}	

\begin{proof}[\textbf{Proof}] Let $\mathscr{R}=\{s<p:\gcd(s,p-1)=1\}$ and set the additive character $\psi(s) =e^{i 2\pi  as/p}\in \C$. Consider a fixed element $u=\tau^s\in \F_p$ and its discrete logarithm $\log_{\tau}u=s$. As the index varies over the set $s\in \mathscr{R}$ ranges over the integers relatively prime to $\varphi(p)=p-1$, the equation 
	\begin{equation}\label{eq9977FF.300DFL1}
		a=s- \log_{\tau}u=0
	\end{equation} 
	has a unique solution $s=\log_{\tau}u$ if and only if the fixed element $u\in \F_p$ is a primitive root. This implies that the inner sum in 	
	\begin{equation}\label{eq9977FF.300DFL2}
		\sum _{\substack{1\leq s< p-1\\\gcd (s,p-1)=1}} \frac{1}{p}\sum _{0\leq t< p} e^{i 2\pi \frac{(s-\log_{\tau}u)t}{p}}=
		\left \{\begin{array}{ll}
			1 & \text{   \normalfont if } \ord_{p} (u)=p-1,  \\
			0 & \text{   \normalfont if } \ord_{p} (u)\ne p-1. \\
		\end{array} \right.
	\end{equation} 
	collapses to $\sum _{0\leq s< p} e^{i 2\pi as/p}=\sum _{0\leq s< p-1} 1=p $. Otherwise, if the element $u-1\in \F_p$ is not a primitive root, then the equation \eqref{eq9977FF.300DFL1} has no solution $s\in \mathscr{R}$ , and the inner sum in \eqref{eq9977FF.300DFL2} collapses to $\sum _{0\leq s< p-1} e^{i 2\pi as/p}=0$,
	this follows from the geometric series formula $\sum_{0\leq n\leq  N-1} w^n =(w^N-1)/(w-1)$, where $w=e^{i 2\pi a/p}\ne1$ and $N=p$. 
	This completes the verification.	 
\end{proof}
\section{Results on Exponential Sums} \label{S9933ERP}\hypertarget{S9933ERP}
The exponential sums of interest in this analysis are presented in this section. 
\begin{lem}   \label{lem9933ERP.120}\hypertarget{lem9933ERP.120}  Let \(p\geq 2\) be a large prime and let $x\leq p$. If $\tau $ is a primitive root modulo $p$ and $a \in [1, p-1]$, then
	\begin{equation}
		\sum_{\substack{1\leq n\leq x\\\gcd(n,p-1)=1}} e^{\frac{i2\pi a \tau^n}{p}} \ll  p^{1/2+\delta}(\log p)^2 ,
	\end{equation} 
	where $\delta>0 $ is a small real number. 
\end{lem} 
\begin{proof}[\textbf{Proof}]The complete proof appears in {\color{red}\cite[Theorem 4.3]{CN2017}}. 
\end{proof}
Assuming $a<x=O(p^{\varepsilon})$ is not a primitive root, the exponential sum occurring in the error term has the trivial evaluation
	\begin{equation}\label{eq9955P.230a} 
\sum_{1\leq b\leq  p-1}	 e^{-i2\pi \frac{ab}{p}}	\sum_{\substack{1\leq n\leq p-1\\\gcd(n,p-1)=1}} e^{\frac{i2\pi b \tau^n}{p}} =-\varphi(p-1).
	\end{equation} 
A nontrivial upper bound is estimated in the next result.
\begin{lem}   \label{lem9933ERP.230V}\hypertarget{lem9933ERP.230V}  Let \(p\geq 2\) be a large prime. If $\tau $ be a primitive root modulo $p$ and $a<x=O(p^{\varepsilon})$ is not a primitive root, then
	\begin{equation}\label{eq9955P.230c} 
	\Bigg|\widehat{V(a)}\Bigg|=	\Bigg|\sum_{1\leq b\leq  p-1}	 e^{-i2\pi \frac{ab}{p}}	\sum_{\substack{1\leq n\leq p-1\\\gcd(n,p-1)=1}} e^{\frac{i2\pi b \tau^n}{p}}\Bigg| = O(p^{1/2+\delta} (\log p)^2)\nonumber,
	\end{equation} 
where $\delta>0$, $\varepsilon>0$ are small numbers and the implied constant is independent of $ a,b \in [1, p-1]$. 	
\end{lem} 

\begin{proof}[\textbf{Proof}]The complete proof appears in {\color{red}\cite[Lemma 4.1]{CN2017}}. 
\end{proof}

The analysis of this class of exponential sums has a growing literature, see \cite{BS2007}, \cite{GD2012}, \cite{CC2008}, et alii.
\section{Fibers and Multiplicities Result} 
The multiplicities of certain values occurring in the estimate of the error term $E(z)$ are computed in this section. 
\begin{lemma}  \label{lem9955P.300S}\hypertarget{lem9955P.300S} Let $p$ be a prime, let $ z<p$ and let $\tau\in \F_p$ be a primitive root in the finite field $\F_p$.  Define the maps
	\begin{equation}\label{eq9955P.300-m}
		\alpha(n,u)\equiv (\tau ^n-u)\bmod p\quad \text{ and } \quad 
		\beta(a,b)\equiv ab\bmod p.
	\end{equation}	
	Then, the fibers $\alpha^{-1}(m)$ and $\beta^{-1}(m)$ of an element $0\ne m\in \F_p$
	have the cardinalities 
	\begin{equation}\label{eq9955P.300-f}
		\#	\alpha^{-1}(m)\leq z-1\quad \text{ and }\quad \#\beta^{-1}(m)=	z
	\end{equation}
	respectively.
\end{lemma}
\begin{proof}Let $\mathscr{R}=\{n<p:\gcd(n,p-1)=1\}$. Given a fixed $u\in [2,z]$, the map 
	\begin{equation}\label{eq9955P.300-m1}
		\alpha:[2,z]\times \mathscr{R}\longrightarrow\F_p\quad  \text{ defined by }\quad  \alpha(n,u)\equiv (\tau ^n-u)\bmod p,
	\end{equation}
	is one-to-one. Here, the map $n\longrightarrow\tau^n \bmod p$ is a permutation the nonzero elements of the finite field $\F_p$, and the map $n\longrightarrow(\tau ^n-u)\bmod p$ is a shifted permutation, see {\color{red}\cite[Chapter 7]{LN1997}} for details. Thus, as $u\in [2,z]$ varies, each value $m=\alpha(n,u)$ is repeated up to $z-1$ times. Moreover, the premises no primitive root $u\leq z$ implies that $m=\alpha(n,u)\ne0$. This verifies that the fiber
	\begin{align}\label{eq9955P.300-f1}
		\#	\alpha^{-1}(m)&=	\#\{(n,u):m\equiv (\tau ^n-u)\bmod p:2\leq u\leq z \text{ and }\gcd(n,p-1)=1\}\nonumber\\[.3cm]
		&\leq z-1.
	\end{align}		
	Similarly, given a fixed $a\in [1,z]$, the map 
	\begin{equation}\label{eq9955P.300-m2}
		\beta:[1,z]\times [1,p-1]\longrightarrow\F_p\quad  \text{ defined by }\quad  \beta(a,b)\equiv ab\bmod p,
	\end{equation}
	is one-to-one. Here the map $b\longrightarrow ab \bmod p$ permutes the elements of the finite field $\F_p$. Thus, as $a\in [1,z]$ varies, each value $m=\beta(a,b)$ is repeated exactly $z$ times. This verifies that the fiber
	\begin{align}\label{eq9955P.300-f2}
		\#	\beta^{-1}(m)&=	\#\{(a,b):m\equiv ab\bmod p:1\leq a\leq z \text{ and }1\leq b< p\}\nonumber\\[.3cm]
		&=z
	\end{align}
	
	Now each value $m=\alpha(n,u)\ne0$ (of multiplicity up to $(z-1)$ in $	\alpha^{-1}(m)$), is matched to $m=\alpha(n,u)=\beta(a,b)$ for some $(a,b)$, (of multiplicity exactly $z$ in $	\beta^{-1}(m)$). These data prove that $\# \alpha^{-1}(m)\leq\# \beta^{-1}(m)$. Notice that  the inequality follows because the map \eqref{eq9955P.300-m1} is not surjective. For example $m=0\in \beta^{-1}(m)$ has multiplicity $z>0$, but $m=0\not\in \alpha^{-1}(m)$ has multiplicity $0$.
\end{proof}

\section{Evaluation of the Main Term} \label{S3355T}\hypertarget{S3355T}
The main term $M(x)=M(x,S_{00})$ occurring in \eqref{eq3355.400m} in \hyperlink{S3355M-T}{Section} \ref{S3355M-T} is evaluated in this Section.
\begin{lemma} \label{lem3355.300T}\hypertarget{lem3355.300T}  If $p>1$ is a large prime number and $x<p$ is a real number, then 
	\begin{equation}
	\sum _{2 \leq u\leq x} \frac{1}{p^2}\sum_{\substack{1\leq n_1\leq p-1\\\gcd(n_1,p-1)=1,}} \sum_{\substack{1\leq n_2\leq p-1\\\gcd(n_2,p-1)=1}}  1
		= \left (\frac{\varphi(p-1)}{p} \right )^2 \cdot  x+O(1) \nonumber,
	\end{equation}
	where $\varphi(n)$ is the Euler totient function.
\end{lemma}

\begin{proof}[\textbf{Proof}] The number of relatively prime integers $n<p$ coincides with the values of the totient function. A routine rearrangement gives 
	\begin{eqnarray}\label{eq3355.300f}
		\sum _{2 \leq u\leq x} \frac{1}{p^2}\sum_{\substack{1\leq n_1\leq p-1\\\gcd(n_1,p-1)=1,}} \sum_{\substack{1\leq n_2\leq p-1\\\gcd(n_2,p-1)=1}}  1 
		&=&\sum _{2 \leq u\leq x} \frac{\varphi(p-1)^2}{p^2}\\[.3cm]
		&=& \left (\frac{\varphi(p-1)}{p} \right )^2 \cdot  x+O(1)\nonumber.
	\end{eqnarray} 
\end{proof}

\section{Error Terms for Dimension $k=1$} \label{S9955LAP-E}\hypertarget{S9955LAP-E}
A nontrivial upper bound of the error term in dimension $k=1$ is computed in this section. To achieve that the error term is partitioned as $E(x)=E_{0}(x)+E_{1}(x)$. The upper bound of the first term $E_0(x)$ for $n\leq p/x$ is derived using geometric summation/sine approximation techniques, and the upper bound of the second term $E_1(x)$ for $p/x\leq n\leq p$ is derived using exponential sums techniques.
\begin{lem}  \label{lem9955SD.300E}\hypertarget{lem9955SD.300E}  Let $p>1$ be a large prime number and let $x=O(p^{\varepsilon})$ be a real number and $\varepsilon>0$ is a small number. If there is no primitive root $u\leq x$ then 
	\begin{equation}\label{eq9955SD.300b}
		\sum_{2 \leq u\leq x}
		\frac{1}{p}\sum_{\gcd(n,p-1)=1,} \sum_{ 0<s\leq p-1} \psi \left((\tau ^n-u)s\right) \ll (\log p)(\log x)\nonumber, 
	\end{equation} 
	where $\psi(s)=e^{i 2 \pi ks/p}$ with $0< k<p$, is an additive character.
\end{lem}

\begin{proof}[\textbf{Proof}] The product of a point $(a,b)\in [1,x]\times [1,p/x)$ satisfies $ab<p$. This leads to the partition $[1,p/x)\cup[p/x,p)=[1,p-1)$, which is suitable for the sine approximation $ab/p\ll\sin(\pi ab/p)\ll ab/p$ for $|ab/p|<1$ on the first subinterval $[1,p/x)$. Thus, consider the partition of the triple finite sum
	\begin{eqnarray} \label{eq9955P.300k}
		E(x)&=& \sum_{2 \leq u\leq x}
		\frac{1}{p}\sum_{\substack{n\leq p-1\\\gcd(n,p-1)=1}} \sum_{ 1\leq s\leq p-1} e^{i2\pi \frac{(\tau ^n-u)s}{p}}   \\[.3cm]
		&= & \sum_{2 \leq u\leq x}
		\frac{1}{p}\sum_{\substack{n< p/x\\\gcd(n,p-1)=1}} \sum_{ 1\leq s\leq p-1} e^{i2\pi \frac{(\tau ^n-u)s}{p}} + \sum_{2 \leq  u\leq x}
		\frac{1}{p}\sum_{\substack{p/x\leq n\leq p-1\\\gcd(n,p-1)=1}} \sum_{ 1\leq s\leq p-1} e^{i2\pi \frac{(\tau ^n-u)s}{p}} \nonumber\\[.3cm]
		&=&E_{0}(x)\;+\;E_{1}(x) \nonumber.
	\end{eqnarray} 
	The first suberror term $E_0(x)$ is estimated in \hyperlink{lem9955SD.700}{Lemma} \ref{lem9955SD.700} and the second suberror term $E_1(x)$ is estimated in \hyperlink{lem9955SD.750}{Lemma} \ref{lem9955SD.750}.  Summing these estimates yields
	\begin{eqnarray} \label{eq9955SD.300u4}
		E(x)&=& E_{0}(x)\;+\;E_{1}(x)   \\
		&\ll& (\log x)(\log p)\;+\;\frac{(\log p)^2}{p^{1/2-\delta}}\cdot x\nonumber\\[.12cm]
		&\ll& (\log x)(\log p)\nonumber,
	\end{eqnarray}
since $x=O(p^{\varepsilon})$, where $\varepsilon>0$ and $\delta>0$ are small real numbers. This completes the estimate of the error term.
\end{proof}

\begin{lem}   \label{lem9955SD.700}\hypertarget{lem9955SD.700}  Let $p>1$ be a large prime number and let $x=O(p^{\varepsilon})$ be a real number and $\varepsilon>0$ is a small number. If there is no primitive root $u\leq x$ then 
	\begin{equation} 
		E_0(x)= \sum_{2 \leq u\leq x}
		\frac{1}{p}\sum_{\substack{1\leq n< p/x\\\gcd(n,p-1)=1}} \sum_{ 1\leq s\leq p-1} e^{i2\pi \frac{(\tau ^n-u)s}{p}}=O\left( (\log x)(\log p)\right) .
	\end{equation} 
\end{lem} 
\begin{proof}[\textbf{Proof}] To apply the geometric summation/sine approximation techniques, the subsum $E_0(x)$ is partition as follows.
	\begin{eqnarray} \label{eq9955SD.700l}
		E_0(x)&=& \sum_{2 \leq u\leq x}
		\frac{1}{p}\sum_{\substack{1\leq < p/x\\\gcd(n,p-1)=1}} \sum_{ 1\leq s\leq p-1} e^{i2\pi \frac{(\tau ^n-u)s}{p}}   \\[.3cm]
		&= & \sum_{2\leq u\leq x}
		\frac{1}{p}\sum_{\substack{1\leq n< p/x\\\gcd(n,p-1)=1}} \left( \sum_{ 1\leq s\leq p/2} e^{i2\pi \frac{(\tau ^n-u)s}{p}}+ \sum_{ p/2<s\leq p-1} e^{i2\pi \frac{(\tau ^n-u)s}{p}}\right) \nonumber\\[.3cm]
		&=&E_{0,0}(x)\;+\;E_{0,1}(x) \nonumber.
	\end{eqnarray} 
	
Now, a geometric series summation of the inner finite sum in the first term yields
	\begin{eqnarray} \label{eq9955SD.700m}
		E_{0,0}(x)&=& \sum_{2 \leq u\leq x}
		\frac{1}{p}\sum_{\substack{1\leq n< p/x\\\gcd(n,p-1)=1}}  \sum_{ 1\leq s\leq p/2} e^{i2\pi \frac{(\tau ^n-u)s}{p}}  \\[.3cm]
		&=&   	\frac{1}{p} \sum_{2 \leq u\leq x}\sum_{\substack{1\leq n<p/x\\\gcd(n,p-1)=1}}   \frac{e^{i2\pi (\frac{\tau ^n-u}{p})(\frac{p}{2}+1)}-e^{i2\pi \frac{(\tau ^n-u)}{p}}}{1-e^{i2\pi \frac{(\tau ^n-u)}{p}}} \nonumber\\[.3cm]
		&\leq&   	\frac{1}{p} \sum_{2 \leq u\leq x}\sum_{\substack{1\leq n< p/x\\\gcd(n,p-1)=1}}   \frac{2}{|\sin\pi(\tau ^n-u)/p|} \nonumber,
	\end{eqnarray} 
	see {\color{red}\cite[Chapter 23]{DH2000}} for similar geometric series calculation and estimation. The last line in \eqref{eq9955SD.700m} follows from the hypothesis that $u$ is not a primitive root. Specifically, $0\ne\tau^n-u\in \F_p$ for any $n \in\{n<p:\gcd(n,p-1)=1\}$ and any $u\leq x=O(p^{\varepsilon})$. Utilizing \hyperlink{lem9955P.300S}{Lemma} \ref{lem9955P.300S}, the first term has the upper bound
	\begin{eqnarray} \label{eq9955SD.700u1}
		E_{0,0}(x)&=&\frac{1}{p} \sum_{2 \leq u\leq x}\sum_{\substack{1\leq n<p/x\\\gcd(n,p-1)=1}}   \frac{2}{|\sin\pi(\tau ^n-u)/p|}\\	[.3cm]
		&\ll&  	\frac{2}{p} \sum_{1\leq a\leq x}\sum_{1\leq b< p/x}   \frac{1}{|\sin\pi ab/p|}\nonumber\\	[.3cm]
		&\ll&  	\frac{2}{p} \sum_{1\leq a\leq z}\sum_{1\leq b< p}   \frac{p}{\pi ab} \nonumber\\	[.3cm]
		&\ll& 	\sum_{1\leq a\leq x}\frac{1}{a}\sum_{1\leq b< p}   \frac{1}{b} \nonumber\\	[.3cm]
		&\ll& (\log x)(\log p)\nonumber,
	\end{eqnarray}
	where $ab<p$ and $|\sin\pi ab/p|\ne0$ since $p\nmid ab$. Similarly, the second term has the upper bound
	\begin{eqnarray} \label{eq9955SD.700v}
		E_{0,1}(x)&=& \sum_{2 \leq u\leq x}
		\frac{1}{p}\sum_{\substack{1\leq n<p/x\\\gcd(n,p-1)=1}}  \sum_{ p/2<s\leq p-1} e^{i2\pi \frac{(\tau ^n-u)s}{p}}  \\[.3cm]
		&=&   	\frac{1}{p} \sum_{2 \leq u\leq x}\sum_{\substack{1\leq n<p/x\\\gcd(n,p-1)=1}}   \frac{e^{i2\pi \frac{(\tau ^n-u)}{p}}-e^{i2\pi (\frac{\tau ^n-u}{p})(\frac{p}{2}+1)}}{1-e^{i2\pi \frac{(\tau ^n-u)}{p}}} \nonumber\\[.3cm]
		&\leq&   	\frac{1}{p} \sum_{2 \leq u\leq x}\sum_{\substack{1\leq n<p/x\\\gcd(n,p-1)=1}}  \frac{2}{|\sin\pi(\tau ^n-u)/p|} \nonumber\\[.3cm]
		&\ll&  (\log x)(\log p)\nonumber.
	\end{eqnarray}
	This is computed in the way as done in \eqref{eq9955SD.700m} to \eqref{eq9955SD.700u1}, mutatis mutandis. Thus, \\		
	\begin{equation}
		E_{0}(x)	=E_{0,0}(x)\;+\;E_{0,1}(x)\ll  (\log x)(\log p).
	\end{equation}
This completes the verification.	
\end{proof}
\begin{lem}   \label{lem9955SD.750}\hypertarget{lem9955SD.750} Let $p>1$ be a large prime number and let $x=O(p^{\varepsilon})$ be a real number and $\varepsilon>0$ is a small number. If there is no primitive root $u\leq x$ then 
	\begin{equation} 
		E_{1}(x)=\sum_{2 \leq u\leq x}	\frac{1}{p}\sum_{\substack{p/x\leq n\leq p-1\\\gcd(n,p-1)=1}}  \sum_{ 1\leq s\leq p-1} e^{i2\pi \frac{(\tau ^n-u)s}{p}}  = 	 O\left(\frac{(\log p)^2}{p^{1/2-\delta}}\cdot x \right) \nonumber,
	\end{equation} 
	where $\delta>0$ is a small real number. 	
\end{lem} 
\begin{proof}[\textbf{Proof}] To determine a nontrivial upper bound, rearrange the triple finite sum as
	\begin{eqnarray} \label{eq9955SD.750d}
		E_{1}(x)&=&\sum_{2 \leq u\leq x}
		\frac{1}{p}\sum_{\substack{p/x\leq n\leq p-1\\\gcd(n,p-1)=1}}  \sum_{ 1\leq s\leq p-1} e^{\frac{i2\pi (\tau ^n-u)s}{p}}\\	[.3cm]
		&=&  \frac{1}{p} \sum_{2 \leq u\leq x ,} \sum _{1\leq s\leq p-1}e^ {\frac{-i2\pi us}{p}} \sum_{\substack{p/x\leq n\leq p-1\\\gcd(n,p-1)=1}}  e^ {\frac{i2 \pi s\tau ^{n}}{p}}  \nonumber\\[.3cm]
		&=&  \frac{1}{p} \sum_{2 \leq u\leq x,}\sum _{1\leq s\leq p-1}e^ {\frac{-i2\pi us}{p}} \left(\sum_{\substack{1\leq n\leq p-1\\\gcd(n,p-1)=1}}  e^ { \frac{i2 \pi s\tau ^{n}}{p}}- \sum_{\substack{1\leq n\leq p/x\\\gcd(n,p-1)=1}}  e^ { \frac{i2 \pi s\tau ^{n}}{p}}  \right)  \nonumber.
	\end{eqnarray}
Taking absolute value and an application of \hyperlink{lem9933ERP.230V}{Lemma} \ref{lem9933ERP.230V} to the double inner exponential (finite Fourier transform) yield
\begin{eqnarray} \label{eq9955SD.750f}
|E_{1}(x)|
&\ll&  \frac{1}{p} \sum_{2 \leq u\leq x} \left|\sum _{1\leq s\leq p-1}e^ {\frac{-i2\pi us}{p}} \sum_{\substack{1\leq n\leq p-1\\\gcd(n,p-1)=1}}  e^ { \frac{i2 \pi s\tau ^{n}}{p}} \right| \\[.3cm]
&\ll&   \frac{1}{p} \sum_{2 \leq u\leq x}
\left|p^{1/2+\delta}(\log p)^2\right| \nonumber\\[.3cm]
&\ll&   \frac{(\log p)^2}{p^{1/2-\delta}}\cdot x \nonumber,
\end{eqnarray}
since $x=O(p^{\varepsilon})$, where $\varepsilon>0$ and $\delta>0$ are small real numbers.
\end{proof}

\section{Error Terms for Dimension $k=2$} \label{S3355DD-E}\hypertarget{S3355DD-E}
The proof of the error term for consecutive pairs of primitive roots in \eqref{eq3355.400m}, (two dimensional) is reduced to the proof for error term for a single element (one dimension). Roughly speaking, the estimate for the one dimension case, proved in \hyperlink{S9955LAP-E}{Section} \ref{S9955LAP-E}, is employed here to complete the estimate for the two dimensional case. The subsets 
\begin{align}
S_{01}, \; S_{10},\; S_{11} \subset [0,p-1]\times [0,p-1]\}
\end{align}
occurring in the partition of $E(x)$ in \eqref{eq3355DD.350Vb} are defined in \eqref{eq3355.400p}. 

\begin{lem}  \label{lem33355DD.350V}\hypertarget{lem3355DD.350V} Let $p>1$ be a large prime number and let $x=O(p^{\varepsilon})$ be a real number and $\varepsilon>0$ is a small number. If the pair $u>1$ and $u+1$ are not consecutive primitive roots mod $p$ then 
	\begin{eqnarray}\label{eq3355DD.350Vb}
E(x)&=&\frac{1}{p^2}\sum _{2 \leq u\leq x}\sum_{\substack{1\leq n_1\leq p-1\\\gcd(n_1,p-1)=1}} \sum_{ \substack{0\leq s_1\leq p-1\\s_1=s_2\ne0}} \psi \left((\tau ^{n_1}-u)s_1\right)    \nonumber\\[.3cm] 
&&\hskip 1.5002 in \times \sum_{\substack{1\leq n_2\leq p-1\\\gcd(n_2,p-1)=1}} \sum_{ \substack{0\leq s_2\leq p-1\\s_1=s_2\ne0}} \psi \left((\tau ^{n_2}-u-1)s_2\right) \nonumber\\[.3cm] 
&=&O\left ((\log x)^2(\log p)^2\right )\nonumber, 
	\end{eqnarray} 
	where the notation $s_1=s_2\ne0$ denotes the exclusion of the point $(s_2,s_1)=(0,0)$.
\end{lem}
\begin{proof}[\textbf{Proof}]The error term is rewritten as a sum of three suberror terms. Summing the estimates of these subterms computed in \hyperlink{lem3355DD.350A}{Lemma} \ref{lem33355DD.350A} to \hyperlink{lem3355DD.350C}{Lemma} \ref{lem33355DD.350C} yields
	\begin{eqnarray}\label{eq3355DD.350Vd}
E(x)&=&E(x,S_{01})\;+\; E(x,S_{10})\;+\; E(x,S_{11})\\[.3cm] 
&\ll&0\;+\;0\;+\;(\log x)^2(\log p)^2 \nonumber\\[.3cm] 
&\ll&(\log x)^2(\log p)^2\nonumber. 
	\end{eqnarray} 
	This completes the estimate of the error term.
\end{proof}
\begin{lem}  \label{lem33355DD.350A}\hypertarget{lem3355DD.350A} Let $p>1$ be a large prime number and let $x=O(p^{\varepsilon})$ be a real number and $\varepsilon>0$ is a small number. If the pair $u>1$ and $u+1$ are not consecutive primitive roots mod $p$ then 
	\begin{align}\label{eq3355DD.350Ab}
E(x,S_{01})&=		\frac{1}{p^2}\sum _{2 \leq u\leq x}\sum_{\substack{1\leq n_1\leq p-1\\\gcd(n_1,p-1)=1}} \sum_{ 1\leq s_1\leq p-1} \psi \left((\tau ^{n_1}-u)s_1\right)    \nonumber\\[.3cm] 
&\hskip 1.5002 in \times \sum_{\substack{1\leq n_2\leq p-1\\\gcd(n_2,p-1)=1}} \sum_{ 0\leq s_2\leq p-1} \psi \left((\tau ^{n_2}-u-1)s_2\right) \nonumber\\[.3cm] 
&=0\nonumber. 
	\end{align} 
\end{lem}
\begin{proof}[\textbf{Proof}] Replace $\psi(s)=e^{i 2 \pi ks/p}$. Use the hypothesis $\tau ^{n_2}-u-1\ne0$ for $u\leq x$ to evaluate the double sum indexed by $s_2$, that is,
\begin{equation}
\sum_{ 0\leq s_2\leq p-1} e^{\frac{i2\pi \left(\tau ^{n_1}-u-1\right)s_2}{p}}=0.
\end{equation}
This step yields
\begin{eqnarray}\label{eq3355DD.350Ad}
E(x,S_{01})&=&		\frac{1}{p^2}\sum _{2 \leq u\leq x}\sum_{\substack{1\leq n_1\leq p-1\\\gcd(n_1,p-1)=1}} \sum_{ 1\leq s_1\leq p-1} e^{\frac{i2\pi \left(\tau ^{n_1}-u\right)s_1}{p}} \sum_{\substack{1\leq n_2\leq p-1\\\gcd(n_2,p-1)=1}} 0\nonumber \\[.3cm] 
&=&0\nonumber.
\end{eqnarray}
\end{proof}

\begin{lem}  \label{lem33355DD.350B}\hypertarget{lem3355DD.350B} Let $p>1$ be a large prime number and let $x=O(p^{\varepsilon})$ be a real number and $\varepsilon>0$ is a small number. If the pair $u>1$ and $u+1$ are not consecutive primitive roots mod $p$ then 
\begin{eqnarray}	\label{eq3355DD.350Bb}
E(x,S_{10})&=&		\frac{1}{p^2}\sum _{2 \leq u\leq x}\sum_{\substack{1\leq n_1\leq p-1\\\gcd(n_1,p-1)=1}} \sum_{ 0\leq s_1\leq p-1} \psi \left((\tau ^{n_1}-u)s_1\right)    \nonumber\\[.3cm] 
&&\hskip 1.5002 in \times \sum_{\substack{1\leq n_2\leq p-1\\\gcd(n_2,p-1)=1}} \sum_{ 1\leq s_2\leq p-1} \psi \left((\tau ^{n_2}-u-1)s_2\right) \nonumber\\[.3cm] 
&=&0\nonumber.
\end{eqnarray}
\end{lem}
\begin{proof}[\textbf{Proof}]Replace $\psi(s)=e^{i 2 \pi ks/p}$. Use the hypothesis $\tau ^{n_1}-u\ne0$ for $u\leq x$ to evaluate the double sum indexed by $s_1$, that is,
\begin{equation}
\sum_{ 0\leq s_1\leq p-1} e^{\frac{i2pi \left(\tau ^{n_1}-u\right)s_1}{p}}=0.
\end{equation}
This step yields
\begin{eqnarray}\label{eq3355DD.350Bd}
E(x,S_{10})&=&		\frac{1}{p^2}\sum _{2 \leq u\leq x}\sum_{\substack{1\leq n_1\leq p-1\\\gcd(n_1,p-1)=1}} 0 \sum_{\substack{1\leq n_2\leq p-1\\\gcd(n_2,p-1)=1}} \sum_{ 1\leq s_2\leq p-1} e^{\frac{i2pi \left(\tau ^{n_2}-u-1\right)s_2}{p}}\nonumber \\[.3cm] 
&=&0.
\end{eqnarray}

\end{proof}
\begin{lem}  \label{lem33355DD.350C}\hypertarget{lem3355DD.350C} Let $p>1$ be a large prime number and let $x=O(p^{\varepsilon})$ be a real number, where $\varepsilon>0$ is a small number. If the pair $u>1$ and $u+1$ are not consecutive primitive roots mod $p$ then 
\begin{eqnarray}	\label{eq3355DD.350Cb}
E(x,S_{11})&=&		\frac{1}{p^2}\sum _{2 \leq u\leq x}\sum_{\substack{1\leq n_1\leq p-1\\\gcd(n_1,p-1)=1}} \sum_{ 1\leq s_1\leq p-1} \psi \left((\tau ^{n_1}-u)s_1\right)    \nonumber\\[.3cm] 
&&\hskip 1.5002 in \times \sum_{\substack{1\leq n_2\leq p-1\\\gcd(n_2,p-1)=1}} \sum_{ 1\leq s_2\leq p-1} \psi \left((\tau ^{n_2}-u-1)s_2\right) \nonumber\\[.3cm] 
&=&O\left ((\log x)^2(\log p)^2\right )\nonumber.
\end{eqnarray}
\end{lem}
\begin{proof}[\textbf{Proof}]To compute a nontrivial upper bound, observe that $E(x,S_{11})$ is a function of $x$ and $p$, which is independent of $u$. Consequently,
\begin{equation}
|E(x,S_{11})|\leq \Bigg |\sum _{2 \leq u\leq x}E(x,S_{11})\Bigg |.
\end{equation}
Next, replacing $\psi(s)=e^{i 2 \pi ks/p}$ yields
\begin{eqnarray}\label{eq3355DD.350Cd}
|E(x,S_{11})|&\leq &\Bigg |\sum _{2 \leq u\leq x}E(x,S_{11})\Bigg | \\[.3cm] &\leq &		\Bigg |\sum _{2 \leq u\leq x,}\sum _{2 \leq u\leq x}\frac{1}{p}\sum_{\substack{1\leq n_1\leq p-1\\\gcd(n_1,p-1)=1}} \sum_{ 1\leq s_1\leq p-1} e^{\frac{i2 \pi  \left(\tau ^{n_1}-u\right)s_1}{p}}  \nonumber  \\[.3cm] 
&&\hskip 1.5002 in \times \frac{1}{p}\sum_{\substack{1\leq n_2\leq p-1\\\gcd(n_2,p-1)=1}} \sum_{ 1\leq s_2\leq p-1} e^{\frac{i2 \pi  \left(\tau ^{n_2}-u\right)s_2}{p}}\Bigg |\nonumber\\[.3cm]
&\leq &		\Bigg |\sum _{2 \leq u\leq x}\frac{1}{p}\sum_{\substack{1\leq n_1\leq p-1\\\gcd(n_1,p-1)=1}} \sum_{ 1\leq s_1\leq p-1} e^{\frac{i2 \pi  \left(\tau ^{n_1}-u\right)s_1}{p}} \Bigg | \nonumber  \\[.3cm] 
&&\hskip 1.5002 in \times \Bigg |\sum _{2 \leq u\leq x}\frac{1}{p}\sum_{\substack{1\leq n_2\leq p-1\\\gcd(n_2,p-1)=1}} \sum_{ 1\leq s_2\leq p-1} e^{\frac{i2 \pi  \left(\tau ^{n_2}-u\right)s_2}{p}}\Bigg |\nonumber.
\end{eqnarray}
Applying \hyperlink{lem9955SD.300E}{Lemma} \ref{lem9955SD.300E} to the two independent triple sums yields
\begin{eqnarray}\label{eq3355DD.350Ce}
|E(x,S_{11})|&\ll&	\log x)(\log p)\cdot 	(\log x)(\log p)\nonumber \\[.3cm] 
&\ll&(\log x)^2(\log p)^2\nonumber.
\end{eqnarray}
\end{proof}
\section{Smallest Pairs of Consecutive Primitive Roots} \label{S3355M-T}\hypertarget{S3355M-T}
The determination of an upper bound for the smallest primitive root in arithmetic progressions modulo $p$ is based on a new characteristic function for primitive roots in finite field $\F_p$ developed in \hyperlink{S9955D}{Section} \ref{S9955D}.

\begin{proof}[\textbf{Proof}] (\hyperlink{thm3355DD.850A}{Theorem} \ref{thm3355DD.850A}) Let \(p>2\) be a large prime number and let $x=O(p^{\varepsilon})$ be a real number, where $\varepsilon>0$ is a small number. Suppose the least pair of consecutive primitive roots $u,u+1\in (x,p)$ and consider the sum of the characteristic function over the short interval \([2,x]\), that is, 
	\begin{equation} \label{eq3355.400l}
		N_2(x)=\sum _{2 \leq u\leq x} \Psi (u)\Psi (u+1)=0.
	\end{equation}
	Replacing the characteristic function, \hyperlink{lem9955.200A}{Lemma} \ref{lem9955.200A}, and expanding the nonexistence equation \eqref{eq3355.400l} yield
	\begin{eqnarray} \label{eq3355.400m}
		N_2(x)&=&\sum _{2 \leq u\leq x} \Psi (u)\Psi (u+1)  \\[.3cm]
		&=&\sum _{2 \leq u\leq x} \left (\frac{1}{p}\sum_{\substack{1\leq n_1\leq p-1\\\gcd(n_1,p-1)=1}} \sum_{ 0\leq s_1\leq p-1} \psi \left((\tau ^{n_1}-u)s_1\right) \right ) \nonumber\\[.3cm] 
&&\hskip 1 in \times \left (\frac{1}{p}\sum_{\substack{1\leq n_2\leq p-1\\\gcd(n_2,p-1)=1}}  \sum_{ 0\leq s_2\leq p-1} \psi \left((\tau ^{n_2}-u-1)s_2\right) \right )\nonumber\\[.3cm] 
		&=&\sum _{2 \leq u\leq x} \frac{1}{p}\sum_{\substack{1\leq n_1\leq p-1\\\gcd(n_1,p-1)=1,}} \sum_{\substack{1\leq n_2\leq p-1\\\gcd(n_2,p-1)=1}}  1 \nonumber\\[.3cm] 
&&\hskip 1 in+	\frac{1}{p^2}\sum _{2 \leq u\leq x}\sum_{\substack{1\leq n_1\leq p-1\\\gcd(n_1,p-1)=1}} \sum_{ 0<s_1\leq p-1} \psi \left((\tau ^{n_1}-u)s_1\right)    \nonumber\\[.3cm] 
&&\hskip 1.75 in \times \sum_{\substack{1\leq n_2\leq p-1\\\gcd(n_2,p-1)=1}} \sum_{ 0<s_2\leq p-1} \psi \left((\tau ^{n_2}-u)s_2\right)\nonumber\\[.3cm] 
		&=&M(x)\; +\; E(x)\nonumber,
	\end{eqnarray} 
The domain $[0,p-1]\times [0,p-1]$ is partitioned as a disjoint union of 4 subsets
\begin{align}\label{eq3355.400p}
S_{00}&=\{(0,0)\}\\[.3cm]
S_{01}&=\{(s_2,s_1)\in [0,p-1]\times [1,p-1]\}\\[.3cm]
S_{10}&=\{(s_2,s_1)\in [1,p-1]\times [0,p-1]\}\\[.3cm]
S_{11}&=\{(s_2,s_1)\in [1,p-1]\times [1,p-1]\}. 
\end{align}
The first subset $S_{00}=\{(0,0)\}$ determines the main term $M(x)=M(x,S_{00})$, the evaluation appears in \hyperlink{lem3355.300T}{Lemma} \ref{lem3355.300T}. The remaining three subsets $S_{s_2,s_1}=\{(s_2,s_1)\ne(0,0)\}$ determines the error term $E(x)=E(x, S_{s_2,s_1})$. A nontrivial upper bound is computed in \hyperlink{lem33355DD.350V}{Lemma} \ref{lem33355DD.350V}. \\
Substituting these estimates and $x=(\log p)^2(\log \log p)^5$ yield
\begin{eqnarray} \label{eq9955.400p}
N_2(x)&=&	M(x)\; +\; E(x) \\[.3cm]
&=&\left[ \left (\frac{\varphi(p-1)}{p} \right )^2 \cdot  x+O(1) \right]  +\left[ O\left((\log x)^2(\log p)^2 \right) \right] \nonumber\\[.3cm]
&=& \left (\frac{\varphi(p-1)}{p} \right )^2 \cdot (\log p)^2(\log \log p)^5+O\left((\log p)^2(\log \log p)^2 \right) \nonumber.
	\end{eqnarray} 
	For any prime $p$, the totient function is bounded away from zero, that is,
	\begin{eqnarray}\label{eq9955.400r}
		\frac{p-1}{p}\cdot \frac{\varphi(p-1)}{p-1}&=&\frac{p-1}{p}\prod_{r\mid p-1}\left( 1-\frac{1}{r}\right)
		\\&\gg&\frac{1}{\log\log p}>0 \nonumber,
	\end{eqnarray}	
	where $r\geq2$ ranges over the prime divisor of $p-1$, see {\color{red}\cite[Theorem 6.12]{DP2016}}. Consequently, the main term in \eqref{eq9955.400p} dominates the error term:
	
\begin{eqnarray} \label{eq9955.400v}
N_2(x)&=&\left (\frac{\varphi(p-1)}{p} \right )^2 \cdot (\log p)^2(\log \log p)^5+O\left((\log p)^2(\log \log p)^2 \right)	\\[.3cm]	
&\gg&\left (\frac{1}{\log \log p}\right )^2  \cdot  (\log p)^2(\log \log p)^5+ O\left((\log p)^2(\log \log p)^2 \right) \nonumber\\[.3cm]
		&\gg&(\log p)^2(\log \log p)^3 \nonumber\\[.3cm]
		&>&0 \nonumber
	\end{eqnarray} 
	as $x\to\infty$. Clearly, this contradicts the hypothesis \eqref{eq3355.400l} for all sufficiently large prime numbers $p \geq p_0$. Therefore, there exists a small pair of consecutive primitive roots \begin{equation}
2\leq u, u+1\leq (\log p)^2(\log \log p)^5
	\end{equation}
for all sufficiently large primes $p$.
\end{proof}

\section{Some Numerical Data}\label{S3355NE}
A few examples were compiled to verify the upper bound and to demonstrate the concept and the expected results.

\begin{exa}{\normalfont For the closest prime to $10^{3}$ lower bound stated in \hyperlink{thm3355DD.850A}{Theorem} \ref{thm3355DD.850A} is not meaningful because 
\begin{equation}
p\leq (\log p)^2(\log \log p)^5.
	\end{equation}
However, it demonstrates the abundance of small consecutive pairs of primitive roots.	The parameters are these: 

		\begin{enumerate}[font=\normalfont, label=(\alph*)]			
			\item $\displaystyle p=1009,$ \tabto{8cm}prime,
\item $\displaystyle p-1=2^4\cdot 3^2\cdot 7,$ \tabto{8cm}factorization,			
			\item $\displaystyle x=(\log p)^{2}(\log_2 p)^5=1,294.24,$\tabto{8cm}upper bound,\\
			
			\item $\displaystyle N_2(x)\gg(\log p)^2(\log_2 p)^3\gg346.04,$\tabto{8cm}predicted number in \eqref{eq9955.400v}.
			
		\end{enumerate}
		\vskip .1 in
The predicted total number of pairs in the finite field $\mathbb{F}_p$ is
\begin{align}\label{eq3355NE.120}
N_2(p)&=c_2(p)\left (\frac{\varphi(p-1)}{p-1}\right )^2\cdot (p-1)+O(p^{1/2+\varepsilon})\\[.2cm]
&=\left (\frac{2^3\cdot6\cdot 6}{1008}\right )^2\cdot 1008+\text{ Error}\nonumber \\[.2cm]
&=82+\text{ Error}\nonumber,
\end{align}
where $c_2(p)>0$ is a constant. The actual count is 84 pairs, listed on the \autoref{table:t100-1}.

\begin{table}[H]
\centering
\vspace{2mm}
\small
\begin{tabular}{|c|c|c|c|c|c|}
	\hline
(11, 12) & (134, 135) & (251, 252) & (395, 396) & (588, 589) & (786, 787) \\
(14, 15) & (149, 150) & (263, 264) & (410, 411) & (597, 598) & (795, 796) \\
(34, 35) & (155, 156) & (275, 276) & (419, 420) & (612, 613) & (816, 817) \\
(39, 40) & (173, 174) & (302, 303) & (431, 432) & (621, 622) & (828, 829) \\
(47, 48) & (179, 180) & (305, 306) & (443, 444) & (648, 649) & (834, 835) \\
(53, 54) & (191, 192) & (314, 315) & (482, 483) & (660, 661) & (852, 853) \\
(72, 73) & (212, 213) & (323, 324) & (491, 492) & (669, 670) & (858, 859) \\
(83, 84) & (221, 222) & (338, 339) & (516, 517) & (684, 685) & (873, 874) \\
(107, 108) & (230, 231) & (347, 348) & (525, 526) & (693, 694) & (876, 877) \\
(110, 111) & (236, 237) & (359, 360) & (564, 565) & (702, 703) & (897, 898) \\
(131, 132) & (386, 387) & (576, 577) & (705, 706) & (732, 733) & (900, 901) \\ 
(744, 745) & (756, 757) & (771, 772) & (777, 778) & (924, 925) & (935, 936) \\
(954, 955) & (960, 961) & (968, 969) & (973, 974) & (993, 994) & (994, 995) \\
	\hline
\end{tabular}
\caption{Complete Set of Consecutive of Pairs Primitive Roots mod $p=1009$}
\label{table:t100-1B}
\end{table}

	}
\end{exa}
\vskip .25 in

\begin{exa}{\normalfont For the closest prime to $10^{12}$ lower bound stated in \hyperlink{thm3355DD.850A}{Theorem} \ref{thm3355DD.850A} begins to fall below $\ll p^{1/4}$. This is meaningful because 
\begin{equation}
(\log p)^2(\log \log p)^5<10p^{1/4}
	\end{equation}
surpasses the established upper bound. The parameters are these: 
		\begin{enumerate}[font=\normalfont, label=(\alph*)]
			
			\item $\displaystyle p=1000000000039,$ \tabto{8cm}prime,
\item $\displaystyle p-1=2\cdot 3\cdot 13\cdot17\cdot29\cdot 26005097,$ \tabto{8cm}factorization,			
			\item $\displaystyle x=(\log p)^{2}(\log_2 p)^5= 12669.61,$\tabto{8cm}upper bound,
			
			\item $\displaystyle N_2(x)\gg(\log p)^2(\log_2 p)^3\gg7601.76,$\tabto{8cm}predicted number in \eqref{eq9955.400v}.
		\end{enumerate}
		\vskip .1 in
The predicted total number of pairs in the finite field $\mathbb{F}_p$ is
\begin{align}
N_2(p)&=c_2(p)\left (\frac{\varphi(p-1)}{p-1}\right )^2\cdot (p-1)+O(p^{1/2+\varepsilon})\\[.2cm]
&=\left (\frac{ 279606792192}{1000000000039}\right )^2\cdot1000000000039+\text{ Error}\nonumber \\[.2cm]
&=78179958237+\text{ Error}\nonumber,
\end{align}
where $c_2(p)>0$ is a constant, see \eqref{eq3355NE.100d}. 
The first 114 consecutive pairs of primitive roots 
\begin{equation}
2\leq u, \;u+1\leq 12670
\end{equation}
are listed on \autoref{table:t100-2}. This small sample illustrates the profusion of consecutive pairs of primitive roots smaller than the stated upper bound.
\begin{table}[H]
\centering

\begin{tabular}{|l|l|l|l|l|l|}
\hline
(46, 47) & (47, 48) & (105, 106) & (188, 189) & (299, 300) & (326, 327) \\ \hline
(348, 349) & (371, 372) & (375, 376) & (381, 382) & (401, 402) & (420, 421) \\ \hline
(421, 422) & (422, 423) & (429, 430) & (430, 431) & (431, 432) & (432, 433) \\ \hline
(437, 438) & (456, 457) & (479, 480) & (497, 498) & (542, 543) & (569, 570) \\ \hline
(582, 583) & (660, 661) & (661, 662) & (662, 663) & (688, 689) & (706, 707) \\ \hline
(772, 773) & (773, 774) & (780, 781) & (781, 782) & (782, 783) & (804, 805) \\ \hline
(860, 861) & (885, 886) & (886, 887) & (887, 888) & (906, 907) & (918, 919) \\ \hline
(919, 920) & (923, 924) & (942, 943) & (945, 946) & (964, 965) & (971, 972) \\ \hline
(1005, 1006) & (1006, 1007) & (1018, 1019) & (1034, 1035) & (1038, 1039) & (1050, 1051) \\ \hline
(1055, 1056) & (1091, 1092) & (1092, 1093) & (1110, 1111) & (1111, 1112) & (1126, 1127) \\ \hline
(1131, 1132) & (1149, 1150) & (1153, 1154) & (1154, 1155) & (1204, 1205) & (1226, 1227) \\ \hline
(1227, 1228) & (1228, 1229) & (1236, 1237) & (1253, 1254) & (1319, 1320) & (1349, 1350) \\ \hline
(1356, 1357) & (1365, 1366) & (1366, 1367) & (1389, 1390) & (1454, 1455) & (1464, 1465) \\ \hline
(1465, 1466) & (1477, 1478) & (1478, 1479) & (1484, 1485) & (1487, 1488) & (1504, 1505) \\ \hline
(1535, 1536) & (1536, 1537) & (1537, 1538) & (1541, 1542) & (1554, 1555) & (1560, 1561) \\ \hline
(1578, 1579) & (1593, 1594) & (1643, 1644) & (1656, 1657) & (1695, 1696) & (1696, 1697) \\ \hline
(1701, 1702) & (1739, 1740) & (1770, 1771) & (1775, 1776) & (1786, 1787) & (1787, 1788) \\ \hline
(1788, 1789) & (1801, 1802) & (1814, 1815) & (1821, 1822) & (1857, 1858) & (1889, 1890) \\ \hline
(1908, 1909) & (1909, 1910) & (1910, 1911) & (1919, 1920) & (1946, 1947) & (1968, 1969) \\ \hline
\end{tabular}
\caption{Smallest 114 Consecutive Pairs of Primitive Roots Modulo $10^{12} + 39$}
\label{table:t100-2}
\end{table}
		
	}
\end{exa}
\newpage
\section{Problems}
Several interesting problems of different level of complexities are presented in this section. The range of difficulty ranges from easy to very difficult.

\subsection{Least Consecutive Primitive Roots In Prime Finite Fields}
\begin{exe} { \normalfont 
Verify that $u_1, u_2, \ldots, u_6=2,3,4,5,6,7$ cannot be a 6-tuple of consecutive primitive roots for any prime. Explain the limitation on the maximal $k$-tuple of consecutive primitive roots in a finite field $\F_p$. }
\end{exe}
\vskip .15in 
\begin{exe} { \normalfont 
Assume that a primitive root is an uniformly distributed and independent random variable with probability $\varphi(p-1)/(p-1)$ in a finite field $\F_p$. Show that the prime $p>e^{e^e}$ cannot support $k>\log p$ consecutive primitive roots. For example, for any prime $p$ such that $\log\log\log p>1$.}
\end{exe}
\vskip .15in 
\begin{exe} { \normalfont 
Let $p$ be a prime , let $x\ll(\log p)^b$ and let $z\gg p^a$, where $a>0$ and  $b\geq0$ are constants. Explain why the distribution of primitive roots in small intervals $[2,x]$ and the distribution on larger intervals $[z,2z]$ are different.}
\end{exe}
\vskip .15in
\begin{exe} { \normalfont 
Let $p$ be a prime and let $x<p$ be a real number. What is the maximal run of consecutive primitive roots $2\leq u, u+1, \ldots, u+m \bmod p$? Develop the analytic and algebraic theory .}
\end{exe}

\vskip .15in
\subsection{Least Consecutive Primitive Roots In Finite Rings}
\begin{exe} { \normalfont 
Let $p$ be a prime and let $p^m$ be a prime power with $m\geq2$. Generalize the concept of consecutive pair of primitive roots to the finite ring $\Z/p^m\Z$. }
\end{exe}
\vskip .15in
\begin{exe} { \normalfont 
Let $p$ be a prime and let $p^m$ be a prime power with $m\geq2$. Generalize the concept of consecutive $k$-tuples of primitive roots to the finite ring $\Z/p^m\Z$. What is the maximal run of $k$ consecutive primitive roots $2\leq u, u+1, \ldots, u+m \bmod p^m$? Develop the analytic and algebraic theory .}
\end{exe}
\vskip .15in
\subsection{Small Consecutive Primitive Roots In Finite Field Extensions}
\begin{exe} { \normalfont 
Let $q=p^m$ be a prime power and let $\F_{q}$ be a finite extension of $\F_p$, with $m\geq2$. Generalize the concept of small consecutive pair of primitive roots to the finite fields $\F_{q}$. }
\end{exe}
\vskip .15in
\begin{exe} { \normalfont 
Let $q=p^m$ be a prime power and let $\F_{q}$ be a finite extension of $\F_p$, with $m\geq2$. What is the maximal run of $k$ consecutive primitive roots $2\leq u, u+1, \ldots, u+m \bmod p^m$ in $\F_{q}$? Develop the analytic and algebraic theory .}
\end{exe}



\newpage
\begin{thebibliography}{998}
	

\bibitem{BS2007} Banks, W. D.; Shparlinksi, I. E. \textit{\color{red}Exponential sums with polynomial values of the discrete logarithm. } Unif. Distrib. Theory 2, No. 2, 67--72 (2007). \href{https://zbmath.org/1162.11039}{Zbl 1162.11039}.
	
	
	
	

\bibitem{CL1956} Carlitz, L. \textit{\color{red}Sets of primitive roots.} Compositio Mathematica. 13 (1956), 65--70. \href{https://zbmath.org/0071.26903}{Zbl 0071.26903}.

\bibitem{CN2017}Carella, N. \textit{\color{red}Densities of Primes and Primitive Roots.} \href{http://arxiv.org/abs/1707.06517}{Arxiv 1707.06517}.

\bibitem{CN2019}Carella, N. \textit{\color{red}Configurations of Consecutive Primitive Roots.} \href{http://arxiv.org/abs/1910.02308}{Arxiv 1910.02308}.


\bibitem{CZ1998} Cobeli, C.; Zaharescu, A. \textit{\color{red}
On the distribution of primitive roots mod $p$.}
Acta Arith. 83 (1998), no. 2, 143--153. \href{https://zbmath.org/0892.11003}{Zbl 0892.11003.}

\bibitem{CC2008} Cristian C. \textit{\color{red}On the discrete logarithm problem.} \href{http://arxiv.org/abs/0811.4182}{Arxiv.0811.4182}. 
 

	
	\bibitem{DH2000} Davenport, H. \textit{\color{red}Multiplicative number theory.} Grad. Texts in Math., 74 Springer-Verlag, New York, 2000
	\href{https://mathscinet.ams.org/mathscinet-getitem?mr=MR1790423}{MR1790423}.
	
\bibitem{DH1937} Davenport, H. \textit{\color{red}On Primitive Roots in Finite Fields.} Q. J. Math., Oxf. Ser. 8, 308--312 (137).
\href{https://zbmath.org/0018.10901}{Zbl 0018.10901}.	
	
	\bibitem{DP2016} Dusart, P. \textit{\color{red}Estimates of some functions over primes, without R.H..} Math. Comp. 85 (2016), no. 298, 875--888. \href{http://arxiv.org/abs/1002.0442}{Arxiv 1002.0442},
	\href{https://mathscinet.ams.org/mathscinet-getitem?mr=MR3434886}{MR3434886}.


\bibitem{GD2012} Gibson, D. J. \textit{\color{red}Discrete logarithms and their equidistribution.} Unif. Distrib. Theory 7, No. 1, 147--154 (2012). \href{https://zbmath.org/1313.11005}{Zbl 1313.11005}.
	
	

	

\bibitem{LE1927} Landau, E. \textit{\color{red}Vorlesungen uber Zahlentheorie: Vol.: 2. Aus der analytischen und geometrischen Zahlentheorie.}
Chelsea Publishing Co., New York, 1969, [1927].
\href{https://mathscinet.ams.org/mathscinet-getitem?mr=MR0250844}{MR0250844}.

\bibitem{LD2015} Liu, H.; Dong, H. \textit{\color{red}On the distribution of consecutive square-free primitive roots modulo $ p$.} Czechoslovak Math. J. 65(140) (2015), no. 2, 555--564. \href{https://zbmath.org/1363.11090}{Zbl 1363.11090}.

\bibitem{LN1997} Lidl, R.; Niederreiter, H. \textit{\color{red}Finite fields.} Second edition. Encyclopedia of Mathematics and its Applications, 20. Cambridge University Press, Cambridge, 1997.
\href{https://mathscinet.ams.org/mathscinet-getitem?mr=MR1429394}{MR1429394}.


\bibitem{LR2022} Lemos, A.; Neumann, V.; Ribas, S. \textit{\color{red}On arithmetic progressions in finite fields.} \href{http://arxiv.org/abs/2208.02876}{Arxiv 2208.02876 }.






	
	
	
	
\bibitem{SM1975} Szalay, M. \textit{\color{blue}On the distribution of the primitive roots of a prime.} J. Number Theory  7  (1975), 184--188.	

	
	

\bibitem{TT2013} Tanti, J.; Thangadurai, R. \textit{\color{red}Distribution of residues and primitive roots}. Proc. Indian Acad. Sci. Math. Sci.  123  (2013),  no. 2, 203--211. 


\bibitem{VE1968}Vegh, E. \textit{\color{red}Pairs of consecutive primitive roots modulo a prime}. Proc. Amer. Math. Soc.  19  (1968), 1169--1170.			

\bibitem{WR2001} Winterhof, A. \textit{\color{red}Character sums, primitive elements, and powers in finite fields.} J. Number Theory 91, 2001, no. 1, 153--163. \href{https://mathscinet.ams.org/mathscinet-getitem?mr=MR1869323}{MR1869323}.	
	
\end{thebibliography}
\end{document}